\def \ga {\gamma}
\def \de {\delta}
\def \ra {\Rightarrow}
\def \be {\beta}
\def \al {\alpha}
\def \b {\beta}

\def \ep {\epsilon}
\def \1 {\frac{1}{2}}

\documentclass{article} 
\usepackage{url,amsmath,amsfonts,amssymb,paulmath}

\newcommand{\sub}{\subseteq}
\title{Free structure of factors}
\author{Ted Hurley} 
\date{} 
\begin{document}
\maketitle
\leftmargin=1pt

\begin{abstract} 
Factors  $\frac{X}{Y}$ in a free group $F$ with $Y$ normal in
$X$ are considered. Precise results on 
 the  free
structure of ${Y}$ relative to the free structure of ${X}$
when $\frac{X}{Y}$ is abelian are obtained. Some extensions and
applications  are given as for example to the construction of lower
central factors in general groups. A
collecting process on free generators, 
which gives basic commutator-type free generators for some subgroups, is also
presented. The notion of {\em relative basic commutators} is developed.
\end{abstract}

\section{Introduction}


This paper is concerned with the {\em free structure} of 
factors  $\frac{X}{Y}$ in a free group $F$,  
 by which is meant  the free structure of $Y$
relative to the free structure  of $X$. More precisely  
the free structure of $\frac{X}{Y}$ determines a free basis 
$A\cup B$ for $X$ such that $B\cup C$ is a free basis for $Y$ where
$C$ is a set obtained from $A,B$ in a basic commutator type
construction. The cases where $\frac{X}{Y}$
is abelian is dealt with in detail and some extensions and
applications are given. A collecting process on free generators which   
gives basic commutator-type free generators for some subgroups is also 
presented. 
  

Let $\frac{X}{Y}$ be a factor in a general group $G$
which  is represented as $\phi: G\cong\frac{F}{R}$ with $R$ normal in the
free group $F$. Then $X
\cong \frac{\hat{X}}{R},  Y\cong \frac{\hat{Y}}{R}$ where $\hat{X},
\hat{Y}$ are the images in $F$ of $X, Y$ respectively under 
$\phi$. Thus, in a sense, factors in free groups represent factors in
general groups.

The $n^{th}$ lower central factor of
$F$, $\frac{\gamma_n(F)}{\gamma_{n+1}(F)}$,  is well known to be the
free abelian group on the {\em basic
  commutators}  of weight $n$ formed from the free generators of
$F$; see for example [1] Chapter 4. Suppose then  $G \cong
\frac{F}{R}$ where $R$ is normal in the free group $F$.
The $n^{th}$ lower central factor of $G$ is 
$\frac{\gamma_n(G)}{\gamma_{n+1}(G)}$ and satisfies
$\frac{\gamma_n(G)}{\gamma_{n+1}(G)} \cong
\left(\frac{\gamma_n(F)}{\gamma_{n+1}(F)}\right)/ \left(\frac{R\cap \gamma_n(F)}{R
\cap \gamma_{n+1}(F)}\right)$.

Thus this general lower central factor is the factor group of the
 known (free) abelian 
 factor by the (free) abelian factor $\frac{R\cap \gamma_n(F)}{R \cap
 \gamma_{n+1}(F)}$.
 The structure of the $n^{th}$ lower central
 factors of $G$ is known once the structure of 
 $\frac{R\cap \gamma_n(F)}{R\cap  \gamma_{n+1}()}$ is known. 

Suppose now  $F$ is free on a finite set and $R$
 is  finitely generated as a
 normal subgroup -- that is, $G\cong \frac{F}{R}$ is finitely
 presented. The free structures of $\frac{R}{R\cap \ga_2(F)}$ and
 $\frac{R\cap \ga_2(F)}{R\cap \ga_3(F)}$ are determined, using in the
 latter case what we
 define as {\em relative basic commutators}. Relative (to
 $R$) basic
 commutators can be defined to study general $\frac{R\cap \gamma_n(F)}{R \cap
 \gamma_{n+1}(F)}$.

The free
 structure of $\ga_m (F)$ all the way down to $\ga_n (F)$ for $n > m$ 
is given in [3]. 

 The {\em Schur Multiplicator} of $G  $
is an  abelian factor  $\frac{R \cap \ga_2(F)}{[R,F]}$
 which  is independent of
the free presentation $G \cong \frac{F}{R}$ -- see for example [2]. 
The Multiplicator is isomorphic to
$\left(\frac{R\cap (\ga_2F)}{R'}\right)/ \left(\frac{[R,F]}{R'}\right)$. 
The free structure of 
 $\frac{R\cap \ga_2(F)}{R'}$ is determined. Generators for $\frac{R\cap
 \ga_2(F)}{[R,F]}$ are given 
in terms of free generators of
 $R$. 

\section{Abelian factors in free groups}

Suppose  $A,B$ are two sets where the union $A\cup B$ is
fully ordered  in a way that every member of $A$ precedes every member
of $B$.   Then the {\em `U-construction'} produces the set, $U =
U(A,B)$ say,  which consists of all words of the form
$$[b^\beta,a_1^{\alpha_1},a_2^{\alpha_2},
  \ldots,a_q^{\alpha_q}]$$ where $q \geq 1, b  \in A\cup B, a_i \in A,
  \text{the indices}\,  \beta,\alpha_i = \pm 1, b > a_1 \leq a_2 \leq
  \ldots \leq a_q \,  \, \text{and if} \, b\in B \Rightarrow \beta = +1$.  
  Further the indices are {\em index coherent} which means that if any
two of  the elements are equal then their indices are the same. 

See [3] for further details on such constructions. 

In other words  the {\em $U$-construction}  forms 
commutators of the type: $$[\frac{\,\,B^{+}}{A}, A, A, \ldots, A]$$
$A$ in a  position means that an element of $A\cup A^{-1}$ occurs in
that position; $B^{+}$ means an element of $B$ (with positive sign) 
may occur in that
position -- the  entries are ordered as  normally expected in a
commutator and in addition there is  the condition that equal entries have the
same sign.


Note:  In the case where $B = \phi$, the
empty set, we  actually do get something useful, namely a free
generating set  for the the derived group  of the group generated by
$\{A\}$,  (when $A$ itself is independent). 

There are a number of equivalent  constructions -- see 
[3].

\begin{theorem}\label{thm:one} If $\frac{X}{Y}$ is a free abelian factor then
  there exists a free basis $A \cup B$ for $X$ such that $B \cup U$ is
  a free  basis for $Y$ where the set $U$ is the  $U$-construction set
  formed  from $A\cup B$.
\end{theorem}

Interpret ``$=: $'' as ``has as free basis''  and then 
the theorem can be visualised  as follows: $$\begin{array}{rrrrrrr}  X
  & =: & A &\cup & B & \,&\, \\ Y  & =: & \, & \, & B &\cup &
  [\frac{\,\,B^{+}}{A}, A, A, \ldots, A]$$\end{array}$$ 





\thmref{two} is similar but applies in more general to a factor 
$\frac{X}{Y}$ which  is a finitely generated abelian group and not just
a free abelian group.

Suppose we are given sets $A_1, A_2, B$ with $|A_1| = r$ and  positive
integers $\Gamma = \{\gamma_1, \gamma_2, \ldots
\gamma_r\}$. We now form the {\em restricted U-construction}, $RU =
RU(A_1,A_2,B, \Gamma )$ on $(A_1, A_2, B)$.

The restricted $U$-construction is similar to the normal $U$-construction
except   now we restrict the  number of occurrences in a commutator of
an element in  $A_1$ .

$RU$ consists of all commutators of the following
form. $$[b^\beta,a_1^{\alpha_1}
  ,a_2^{\alpha_2},\ldots,a_q^{\alpha_q}]$$ where $q \geq 1, b \in
A\cup  B, a_i \in A, \, \,\text{the indices}\,\, \beta,\alpha_i = \pm
1,  \, b > a_1 \leq a_2 \leq \ldots \leq a_q \,\,\text{and if}\,\,b\in
B  \Rightarrow \beta = +1$. 

Further the indices are {\em index coherent} which means that if any
two  of the elements are equal then their indices are the same. {\em
  Further if  $x_i$ $\in A_1$  occurs in the commutator its
  length  (= the number of  times it occurs) is $\leq
  \frac{1}{2}\alpha_i$  and if its length is equal to
  $\frac{1}{2}\alpha_i$  then its exponent is $+1$.}

Define $\hat{A_1} = \{x_1^{\alpha_1}, x_2^{\alpha_2},
\ldots,x_r^{\alpha_r} \}$ and $\hat{B} = B \cup \hat{A_1}$. 




\begin{theorem}\label{thm:two} If $\frac{X}{Y}$ is a finitely generated abelian
factor 
which is the direct product of the cyclics $C_{\gamma_i}$ generated by
  $x_i$ for  $1 \leq i \leq r$
  and $t$ infinite cyclic groups generated by $x_j$ for $ r < j \leq
  r+t$ and $\gamma_i / \gamma_{i+1}$ for $ 1 \leq i \leq r-1$.
 Then  there also exists a free basis $A \cup B$ for $X$ with  $A =
  A_1  \cup A_2$ such that $  \hat{B} \cup RU$ is a free
  basis  for $Y$ where the set $RU = RU(A_1,A_2,B,\Gamma)$ is the restricted U-construction
  formed from $(A_1, A_2, \hat{B})$. 
\end{theorem}

We can   
visualise these theorems as follows: Let $\frac{X}{Y}$ be abelian in
the free group $F$. Then: $$\begin{array}{rrrrrrr}  X  & =: & A & \cup
  & B & \,&\,  \\ Y  & =: & \, & \, & B &\cup&  C $$\end{array}$$ 

The set $C$ is of course different in the two theorems. In
\thmref{one}, 
$\frac{X}{Y}$ may be infinitely generated.


A collecting process in free groups is also presented which has
independent interest; the process is set up within the proofs of the theorems. 
\subsection{Proofs}

{\bf \thmref{one}}: {\em If $\frac{X}{Y}$ is a free abelian factor then
  there exists a free basis $A \cup B$ for $X$ such that $B \cup U$ is
  a free  basis for $Y$ where the set $U$ is the  $U$-construction set
  formed  from $A\cup B$.} 

\begin{proof}
To prove \thmref{one} we proceed in three stages:
\begin{enumerate}
\item $X$ has a free generating set $A\cup B$ where $A$ freely
 generates $\frac{X}{Y}$  and each element of $B$ is 
in $Y$.
\item $B \cup C$ is independent.
\item $B \cup C$ generates $Y$.
\end{enumerate} 

We show item 1 initially  when $\frac{X}{Y}$ is finitely
generated. Suppose then $\frac{X}{Y}$ is free abelian on 
$x_1, x_2, \ldots , x_r$.

Then in terms of a free basis $y_1, y_2, \ldots  $ for $X$ we can
write, using only a finite number of  the free generators,  
$$x_i \equiv y_1^{\alpha_{i,1}}y_2^{\alpha_{i,2}} \ldots
y_k^{\alpha_{i,k}}  \mod X' \sub Y$$
where the $\alpha_{i,j} \in \Z$. Then by a series of change of free
(abelian)  variables for $\frac{X}{Y}$ and free variables for 
$X$ we may assume there exists a free abelian basis $x_1, x_2, \ldots
,  x_r$ for $\frac{X}{Y}$ and a free basis for
$X$ so that $x_i \equiv y_i^{\alpha_i} \mod X'$ with $\alpha_i \geq
0$. See [2] Chapter 3. 

Now no $\alpha_i $ can be $0$ as $\frac{X}{Y}$ cannot be generated 
by less than $r$ elements. Also since each $y_i$ can be written in
terms  of the $x_j$ modulo $Y$ it also follows that none of the  
$\alpha$ can be greater than 1. Thus $T = \{y_1, y_2, \ldots, y_r\}$ freely 
generates $\frac{X}{Y}$ and $T$ is part of a free basis, $Q$ say, for $X$.


If $y \in Q, y\notin T$  then $y
\equiv t_1^{\alpha_1}t_2^{\alpha_2} \ldots t_r^{\alpha_r} \mod Y $ with the
$t_i \in T$. Now replace $y$ by  $(t_1^{\alpha_1}t_2^{\alpha_2} \ldots
t_r^{\alpha_r})^{-1}y $ and we see that we can assume that $y \in Y$.

Consider now the case when $\frac{X}{Y}$ is infinitely (countably)
generated.  Choose $T =
t_1, t_2, \ldots, t_i, \ldots $ maximal so that $T$ has the property
  that it is part of a free (abelian) basis for $\frac{X}{Y}$ and is part of a
  free basis, $Q$ say,  for $X$. 

If $T$ freely generates $\frac{X}{Y}$ then we are done. Otherwise we have a
set $T \cup x$ which is part of a free generating set for $\frac{X}{Y}$. Now
  modulo $X'$, $x \equiv y_1^{\alpha_1}y_2^{\alpha_2} \ldots y_s^{\alpha_s}$
  with the $y_i \in Q$  and the $\alpha_i \in \Z$. By changing the
  free generator $x$ of
  $\frac{X}{Y}$ we may assume that none of the elements of $T$ occur in the
  expression for $x$. Then by changes of variables we may assume $x
  \equiv y^\alpha \mod Y$, $\alpha \geq 1, y\notin T$ and  with $T \cup y$ part
  of a free generating set for
  $X$. Since $y$ may be written in terms of the free generators of
  $\frac{X}{Y}$ it is clear that $\alpha$ must be 1. Thus $T \cup y$ is part
  of a free basis for both $\frac{X}{Y}$ and for $X$. 

As with the finitely generated case we may assume, by changing
variables if necessary, that each element of the free generators which
is not in $T$ is in $Y$: If $y \in Q, y\notin T$  then $y
\equiv t_1^{\alpha_1}t_2^{\alpha_2} \ldots t_r^{\alpha_r} \mod Y $ with the
$t_i \in T$. Now replace $y$ by  $(t_1^{\alpha_1}t_2^{\alpha_2} \ldots
t_r^{\alpha_r})^{-1}y $ and we see that we can assume that $y \in Y$.

That $B \cup C$ is independent is shown in [2] Theorem 2.1. We now need to
show  that $B \cup C$ generates $Y$.  To do this we introduce a {\em
  collecting  process} on free generators. 

Suppose $y \in X$ then $y = w(A,B)$, a word
in  $A$ and $B$. In this word {\bf collect elements of $A$ only}.  Then
what happens is the uncollected piece consists of a word in $B$ and
$C$.  Thus we show that:$$y = a_1^{\alpha_1}a_2^{\alpha_2}\dots
a_r^{\alpha_r} \times w(B,C)\,\,\,(**)$$ where the $a_i$ are in $A$ and
the $\alpha_i$ are in $\mathbb{Z}$.  Now if $y\in Y$, $w(B,C) \in Y$
and  the elements of $A$ are independent modulo $Y$, it follows that
all  the $\alpha_i$ are $0$ and thus $y = w(B,C)$.

We now need to show that $y$ can be written in the form $(**)$. Suppose
$b \in B, x\in A\cup A^{-1}$. Then $$bx = xb[b,x]$$ and $$b^{-1}x
=x[b,x]^{-1}b^{-1}$$  These are the fundamental collection formulae.

We may assume that $A$ is finitely generated since we are only
considering  a finite number of elements of $A$ in the expression for
$y$. Set $A = \{a_1, a_2, a_3, \ldots a_n\}$. Proceed by induction on
$n$ to show that $y$ is a product of elements of the required form.

By induction we may assume that $y$ has the form  $$y =
a_1^{\alpha_1}a_2^{\alpha_2}\dots  a_{n-1}^{\alpha_{n-1}} \times
w(B',C')\,\,\,(**)$$where now $B' = B\cup a_n$, $A' = A - \{a_n\}$,
and  $C'$ is the set of elements obtained by performing the $U$
construction  on $(A', B')$. We now collect $a_n$. 

Suppose $c\in C'$, $c$ does not contain $a_n$ (which means $c$ does
not begin with $a_n$) and $a \in \{a_n, a_n^{-1}\} $. Then $ca  =
ac[c,a]$ and $c^{-1}a =  a[c,a]^{-1}c^{-1}$.  Then $[c,a] \in  C$.
Suppose now $c\in C$ with $c$ containing an element of $B$, $a$ as
before  and where now we allow  $c$ to end in $a_n^{\pm 1}$. If the
last  entry of $c$ has the same  sign as $a$ then $[c,a] \in C$. If
last  entry of $c$ has different  sign to $a$ then $c = [c',a^{-1}]$
and $[c,a] =  [c',a^{-1},a] = [c',a^{-1}]^{-1}[c',a]^{-1}$. 
Now $[c',a^{-1}]$ is in a word in $B,C$ and we
then proceed by induction on
the number of  occurrences of $a^{-1}$ in $c$ to show that  $[c',a]$ is a
 product of  elements in  $C$. 

If $c \in C'$ contains $a_n$ or if $c \in C$ with $c$ not involving an
element  of $B$ then $ca
=ac[c,a]$  and $[c,a]$ is now in the commutator subgroup of the group
generated by $A$ and is thus a product of elements of $C$. This
completes  the proof.
\end{proof}
\quad \quad

Suppose now $\frac{X}{Y}$ is a finitely generated abelian section
which is the direct product of cyclic groups $C_{\gamma_1},
C_{\gamma_2},\ldots , C_{\gamma_r}$ and of $t$ infinite cyclic groups,
where $C_{\gamma_i}$ has order  $\gamma_i$ for $ 1 \leq i\leq r$ and
such that $\gamma_1/\gamma_2/\ldots /\gamma_r$.

{\bf \thmref{two}}: {\em  If $\frac{X}{Y}$ is a finitely generated abelian  section
which is the direct product of the cyclics $C_{\gamma_i}$ generated by
  $x_i$ for  $1 \leq i \leq r$
  and $t$ infinite cyclic groups generated by $x_j$ for $ r < j \leq
  r+t$ and $\gamma_i | \gamma_{i+1}$ for $ 1 \leq i \leq r-1$.
 Then  there also exists a free basis $A \cup B$ for $X$ with  $A =
  A_1  \cup A_2$ such that $  \hat{B} \cup RU$ is a free
  basis  for $Y$ where the set $RU = RU(A_1,A_2,B,\Gamma)$ is the restricted U-construction
  formed from $(A_1, A_2, \hat{B})$.}

\begin{proof}
 
To prove this theorem we need:
\begin{enumerate}
\item $X$ has a free generating set $A_1 \cup A_2 \cup B$ where $\frac{X}{Y}$ is
 the direct product of the cyclic groups $C_{\gamma_i}$ with 
 $C_{\gamma_i}$  generated by $x_i \in A_1$, $t$ infinite
 cyclics generated by the elements of $A_2$,  and a set $B$ in which
 each element is in $Y$. 
\item $\hat{B}  \cup RU$ is independent.
\item $\hat{B} \cup RU$ generates $Y$.
\end{enumerate} 

Suppose then the torsion part of $\frac{X}{Y}$ is generated by  
$x_1, x_2, \ldots , x_r$ where $x_i$ has order $\gamma_i$.

Then $x_1 \equiv y_1^{\alpha_1}y_2^{\alpha_2} \ldots y_s^{\alpha_s} \mod X' \sub
Y$  with the $y_i$ in a free generating set for $X$. Then by change of
the variables for $X$ we may assume $x_1 \equiv y^\alpha \mod X' \sub
Y$.

Now $y \equiv x_1^{\delta_1}x_2^{\delta_2}\ldots x_r^{\delta_r} \times a$
with $a$ in the torsion free part of $\frac{X}{Y}$. Putting these together we
get that $x_1 \equiv y^\alpha \mod Y$ and $y \equiv x_1^\delta \mod Y$. From
this it is deduced that $x_1$ and $y$ have the same order modulo $Y$
and that $x_1$ and $y$ generate the same subgroup modulo $Y$. We can
thus replace $x_1$ by $y$ as the generator of the cyclic group of
order $\gamma_1$.

Consider $x_2$ which has order $\gamma_2$. Now
  $x_2 \equiv y_1^{\alpha_1}y_2^{\alpha_2} \ldots y_s^{\alpha_s} \mod
  X' \sub Y$ . Now by replacing $x_2$ by $(y_1^{\alpha_1})^{-1}* x_2$, which also
  has order precisely $\gamma_2$ since $\gamma_1 | \gamma_2$ we see
  that  we may assume that the
  $y_1$ does not appear in this expression for $x_2$. Then as for the
  case $x_1$ we may replace this $x_2$ by a free generator $y_2$ of
  $x$ which has also order $\gamma_2 \mod Y$.

We continue in this way to replace each $x_i, 1 \leq i \leq r$
 by a free generator of $X$ which has also order $\gamma_i$ modulo
 $Y$.

Let $T = \{y_1, y_2, \ldots, y_r\}$ where
$y_i$ generates $C_{\gamma_i}$ in $\frac{X}{Y}$.

Let $X$ then have basis $ Q = T \cup R$. 
The (free abelian) generators of $\frac{X}{Y}$ which have infinite order are
dealt with in the same manner as for the finitely generated free
abelian case above. Suppose we have $x_{r+k}$ of infinite order in
  $\frac{X}{Y}$. First of all write $x_{r+k}$ as a product of the free
generators $Q$ of  $X$ modulo $X'$. If any of the free
generators $y_i \in T$ occurs to the power of $\alpha$ in this
expression then replace $x_{r+k}$ by $y_i^{-\alpha}x_{r+k}$. This new element
also has infinite order and does not contain $y_i$ modulo $X'$. If it contains 
a power of $y_i$ modulo $Y$ then this power must be a multiple of $\gamma_i$; 
in this way we can
ensure that the element of infinite order do not contain any of the
free generators constructed which have finite order modulo $Y$. We
then proceed as for the finitely generated case in the free abelian
case.

Also if $x$ is in the free basis for $X$ which is not one of the $y_1,
\ldots y_r, y_{r+1}, \ldots y_{r+t}$ then as before a change of
variable will ensure this is in $Y$ (as it can be written as product
of the $y_i, 1 \leq i \leq r+t$ modulo $Y$).

Now  $RU = RU(A_1,A_2,B,\Gamma)$ denotes the restricted construction on $(A_1, A_2, \hat{B})$.
The next step is to show that $\hat{B} \cup RU$ is independent.

It is clear that $\hat{B}$ is independent (as $B \cup A_1$ is independent).

We refer to [2] where it is shown that $B \cup U$ is independent and is
also  equivalent to the set $B \cup Z$. 
We now show that every element of $Y$ can be written in terms of $RU$.
We show that for  $x \in X$ then 

$$x = x_1^{\alpha_1}x_2^{\alpha_2}\ldots
x_r^{\alpha_r}x_{r+1}^{\alpha_{r+1}}  \ldots
x_{r+t}^{\alpha_{r+t}}\times  w(\hat{B}, Z)$$ 

where $\alpha_j \in \Z$ with $0 \leq \alpha_i < \delta_i$ for $1 \leq
i  \leq r$.

 Let $F$ be a free group which contains a subset which 
is an
ordered disjoint union $A\cup B$. Several ways of constructing new subsets 
of
$F$ from $A$ and $B$ were defined in [2]. These are


The \lq\lq$Z$-construction\rq\rq\ produces the set, $Z$ say, 
which
consists of all words of one or other of the two forms
{
$$
b^{a_1^{\al_{1}}a_2^{\al_{2}}\ldots a_q^{\al_{q}}} =
a_q^{-\al_{q}}a_{q-1}^{-\al_{q-1}}\ldots a_1^{-\al_{1}}b
a_1^{\al_{1}}a_2^{\al_{2}}\ldots a_q^{\al_{q}}
$$
where $q\ge 1$, $b\in B$, the $a_i$ are members of $A$, each $\al_i=\pm 
1$,\break
$a_1\le a_2\le\ldots$ $\le a_q$ (note that $b>a_1$ is automatically true) 
and the
sequence $a_1^{\al_{1}},a_2^{\al_{2}},\ldots, a_q^{\al_{q}}$ is index-
coherent,}

or
{
$$
\left[b^\b,a_1^{\al_{1}}a_2^{\al_{2}}\ldots
a_p^{\al_{p}}\right]^{a_{p+1}^{\al_{p+1}}a_{p+2}^{\al_{p+2}}\ldots
a_q^{\al_{q}}} =  a_q^{-\al_{q}}a_{q-1}^{-\al_{q-1}}\ldots
a_{p+1}^{-\al_{p+1}}b^{-\b}a_p^{-\al_{p}}a_{p-1}^{-\al_{p-1}}\ldots
a_1^{-\al_{1}}b^\b a_1^{\al_{1}}a_2^{\al_{2}}\ldots a_q^{\al_{q}}
$$
where $1\le p\le q$, $b$ and the $a_i$ are members of $A$, $\b$ and each
$\al_i=\pm 1$, $b>a_1\le a_2\le\ldots\le a_p<b\le a_{p+1}\le 
a_{p+2}\le\ldots\le
a_q$ and the sequence $a_1^{\al_{1}},a_2^{\al_{2}},\ldots, a_q^{\al_{q}},b^\b$ 
is
index-coherent.}

The first construction produces the set $Z_1$ and the second produces
$Z_2$ and then $ Z = Z_1 \cup Z_2$.

We have already seen the $U$ construction which is as follows.

The \lq\lq$U$-construction\rq\rq\ produces the set, $U$ say, 
which
consists of all words of the form
{
$$
\left[b^\b,a_1^{\al_{1}},a_2^{\al_{2}},\ldots,a_q^{\al_{q}}\right]
$$
where $q\ge 1$, $b\in A\cup B$, the $a_i$ are members of $A$, $\b$ and 
each
$\al_i=\pm 1$, $b>a_1\le a_2\le\ldots\le a_q$, the sequence
$a_1^{\al_{1}},a_2^{\al_{2}},\ldots,a_q^{\al_{q}},b^\b$ is index-coherent and
$b\in B\ra\b=+1$.}

It is shown in [2, Theorem 2.1] that the sets $B \cup Z$ and $B \cup U$ are equivalent and if
  $A \cup B$ is independent then so is $B \cup U$.

We now consider the case where $A = A_1 \cup A_2$ and $\hat{B} = B
\cup \hat{A_1}$ and form the restricted $Z$ and $U$
constructions.  
Thus if $x_i \in A_1$ and $\Gamma = \gamma_1, \ldots
,\gamma_r$ then the number of occurrences of $x_i \leq \gamma_i$ and
if then number is actually equal to $\gamma_i$ then the sign of $x_i$
is $+1$. 

Denote the sets produced by the restricted $Z$ and $U$ constructions
by  $\hat{Z}$ and $\hat{U}$ respectively.

It is
clear that $B \cup Z$ is independent (whether or not $Z$ is
restricted).  (We have already noted that $\hat{B}$ is independent.)
It  is thus sufficient to show that $\hat{A_1} \cup Z$
when $Z$  is restricted is independent. But this is clear as this set
is Nielsen reduced - the restriction ensures that cancellation does
not proceed so as to involve the central significant factor. 

Theorem 2.1 of [3] may then be modified to show that $\hat{B} \cup
\hat{Z}$ and $\hat{B} \cup \hat{U}$ are equivalent. 

We now need to show that $Y$ is generated by $\hat{B} \cup
\hat{Z}$. It follows immediately that $Y$ is generated by $\hat{B}
\cup \hat{U}$
and that each of these sets freely generate $Y$.   
We do this by a collection process on $A \cup B$. We show that if $x
\in X$ then 

$$ x = x_1^{\alpha_1} x_2^{\alpha_2} \ldots
x_{r+t}^{\alpha_{r+t}}\times w(\hat{B}, \hat{Z})$$ with $\alpha_i \in
\Z$ where $0 \leq
\alpha_i < \gamma_i$ for $1 \leq i \leq r$ and $w(\hat{B}, \hat{Z})$ is
a word in $ \hat{B} \cup \hat{Z}$

Then if $x \in Y$ all the $\alpha_i$ must be $0$.

Initially $x =
w(A,B)$.   We collect elements of $A$ only but in a restricted
manner. First  we collect elements of $A_1$.  

Suppose $cx^\ep$ occurs with $c > x$ and $x$ has to be collected.
The fundamental collection here is 

$$ cx^ep = xb^{x^\ep}$$

Suppose $ c= b^{x^\de}$ and then we get $b^{x^\de}x^\ep = x^\ep
b^{x^{\de + \ep}}$. Suppose now  $x \in A_1$ is of pseudo-order $\gamma$. 
If $|\de + \ep | < \frac{1}{2} \ga$ then this finishes collection. If
$|\de + \ep | \geq  \frac{1}{2}\ga $ then $ x^\ep
b^{x^{\de + \ep}} = x^ep x^{-\ga} b^{x^{-\ga + \de + \ep}}x^\ga$ when
$\de + \ep > 0$ and   $ x^\ep
b^{x^{\de + \ep}} = x^ep x^{\ga} b^{x^{\ga + \de + \ep}}x^{-\ga}$ when
$\de + \ep \leq 0$. 
In all cases we ensure that the power of $x$ in the conjugate of $b$
occurs less than or equal to $\1 \ga $ and if equal to $\1 \ga$ then
it has positive power.

If $b = x^\alpha$  with $x$ to be collected then
first of all $b = x^\alpha = x[x,\alpha]$. Then $x$ is collected and
it is collected over $[x,\alpha]$ to give elements of $Z_2$ and
consequently elements of $Z_2$ when further elements of $A$ are
collected.

When $x \in A_1$ is fully collected it occurs in the front of elements
of $Z$ in the from $x^\alpha$ with $|\alpha| \leq \frac{1}{2} \ga$ and
if equal to $\frac{1}{2}\ga$ then $\alpha > 0$, (where $\ga$ is the
pseudo-order of $x$).  We now ensure that $x$
occurs in the form $x^\alpha $ before the elements of $Z$  with $ 0
\leq \alpha < \de$ by $x^\alpha = x^{\ga + \alpha}x^\ga$ when $\alpha
< 0$ (and $x^\ga \in Z$).

Thus if $p$ in $X$ then 

$$p = x_1^{\alpha_1}x_2^{\alpha_2}\ldots x_{r+t}^{\alpha_{r+t}}\times
w(\hat{B},\hat{Z})$$
with $0 \leq \alpha_i < \ga_i$ when $x_i \in A_1$ has pseudo order
$\ga_i$.

Then if $p \in Y$ it follows that all the $\alpha_i = 0$ and $p$ is a
word in  $\hat{B}, \hat{Z}$ as required. 
\end{proof}


It is also in certain cases  possible to go constructively down abelian
sections  to get to a group:  for example we could study  {\em metabelian
  section} $\frac{X}{Z}$ if we know the abelian sections $\frac{X}{Y}$
and $\frac{Y}{Z}$.  Having worked on $X$ modulo $Y$ we then   look
at $Y$  modulo $ Z$.The processes are inductive so for example when
 a series of factors are finitely generated it is in theory
possible to work from the top group all the way down. 


\subsection{Lower central factors}\label{sec:five}


Suppose $F$ is freely generated by the finite set $X$ 
and that $R$ is generated as a normal subgroup by
 $ A = \{r_1, r_2,\ldots, r_m\}$. 
\begin{lemma} There exists a set  of free generators
$x_{1}, x_{2}, \ldots , x_{n}$ for $F$ and a set of free  
generators \\$w_1, w_2, \ldots, w_t, w_{t+1}, \ldots $ for $R$ such that
$w_i \equiv x_{i}^{d_{i}}  \mod \gamma_2F  $ for $1 \leq i
\leq t \leq n$ where $d_i \not = 0$ and $w_i \in \ga_2(F)$ for $ i > t$.
\end{lemma}
\begin{proof} Let $\hat{w}_1, \hat{w}_2, \ldots \hat{w}_s$ be the free generators for
$R$ involved in the expressions for $r_1, r_2, \ldots ,r_m$ as words
in the free generators of $R$. Then by [2], Chapter 3 (Theorem 3.5)
there is a set of free generators
$x_{1}, x_{2}, \ldots , x_{n}$ for $F$ and a set $w_1, w_2, \ldots,w_s$
Nielson equivalent to $\hat{w}_1, \hat{w}_2, \ldots , \hat{w}_s$
 such that
$w_i \equiv x_{i}^{d_{i}} \mod \gamma_2F  $ for $1 \leq i
\leq t \leq s$ where $d_i \not = 0$ and $w_i \in \ga_2(F)$ for $ i > s$.
Let $w_1, w_2, \ldots, w_s, w_{s+1}, \ldots  $ denote the free
generators of $R$. Then $w_{s+i}$ for $i \geq 1$ is a word, say
$w(s+i)$, in $w_1, w_2,
\ldots , w_s \mod \ga_2(F)$ since $r_1, r_2, \ldots r_n$ generate $R
\mod [R,F]\subset \ga_2(F)$. Thus  replacing $w_{s+i}$ by
$w_{s+i}w(s+i)^{-1}$ in the free generating set for $R$ 
we may assume $w_{s+i} \in \ga_2(F)$.
\end{proof}


Let $W_1 = w_1, w_2, \ldots, w_t $ and $W_2 = w_{t+1}, \ldots $. Then:

\begin{theorem} $R$ has free basis $W_1 \cup W_2$ and $R\cap \ga_2(F)$
has free basis $W_2 \cup U$ where $U$ is the $U$-construction on 
$W_1\cup  W_2$.
\end{theorem}

All the elements of $U$ except those of the form $[w_i^{\pm 1},w_j^{\pm
1}]$ with $w_i, w_j \in W_1$ are automatically in $\ga_3(F)$. Using
$[a^{-1},b] = [a,b]^{-1}[a,b,a^{-1}]$  and $[a,b^{-1}] =
[a,b]^{-1}[a,b,b^{-1}] $ we may replace $[w_i^{\pm 1},w_j^{\pm
1}]$ where one or both of the signs are $-1$ by
a free generator in $\ga_3(F)$. 

Let $\hat{W}$ be the set $\{[w_i,w_j]\} \in U$. Note that $[w_i,w_j]
\cong [x_i,\_j]^{d_id_j} \mod \ga_3(F)$ and that $[x_i,x_j]$ is a
basic commutator of weight $2$. Now  
 set $W = \hat{W} \cup \{w_{t+1}, \ldots \}$

Then there exists a set  $Q= q_1, q_ 2, \ldots $ equivalent to $W$ such that  
$q_i \cong b_i^{\al_i} \mod \ga_3(F)$ for $1 \leq i \leq s$, $\al_i
\not = 0$ where $\b_1, b_2, \ldots,b_s$ is equivalent to a set of $s$
basic commutators of weight $2$ and $q_i \in \ga_3(F)$. 

Set $Q_1 = q_1, q_2, \ldots , q_s$ and $Q_2 = q_{s+1}, \ldots, $.

Then:
\begin{theorem} $R\cap \ga_2(F)$ has free basis $Q_1 \cup Q_2$ and
$R\cup \ga_3(F)$ has free basis $Q_2 \cup \hat{U}$ where $\hat{U}$ is
the $U$-construction on $Q_1,Q_2$.
\end{theorem}

Call $U_2$ the set of {\em $R$-basic commutators} of weight
$2$ and  $W_1$  the set $R$-basic commutators of weight $1$. 
 It is possible to similarly define $R$-basic commutators of higher
 weight and this is the subject of further work.
 

\begin{theorem} Every element $w$ in $R$ can be written uniquely
in the form

$$ w \equiv r_1^{\alpha_1} r_2^{\alpha_2} \ldots r_t^{\alpha_t} \quad
modulo \quad R \cap \gamma_{3}F$$

where the $r_1, r_2, \ldots ,r_t$ are the $R-$basic commutators of
weights $ \leq 2$ and $r_1 < r_2 < \ldots < r_t$ and the $\alpha_i$
are non-negative integers.

\end{theorem}
\medskip

This process can be continued and we can define a set of
$R-$basic of weight $n$ which will be a basis for $\frac{R \cap
  \gamma_nF}{R \cap \gamma_{n+1}F}$.

This general method follows the
process as given above for the cases $n=2,3$. 
A basic commutator $b$ which
corresponds non-trivially to a free generator $w$ modulo $\gamma_mF\cap R$
is replaced by this $w$ in any further basic commutator which
contains this $b$ as a constituent. The details are omitted here 
but gave rise to the general idea of $R-$basic commutators
of weight $n$.  A Hall-like relative basis theorem then follows:

\begin{theorem} Every element $w$ in $R$ can be written uniquely
  in the form

$$ w \equiv r_1^{\alpha_1} r_2^{\alpha_2} \ldots r_t^{\alpha_t} \quad
  modulo \quad R \cap \gamma_{n+1}F$$

where the $r_1, r_2, \ldots ,r_t$ are the $R-$basic commutators of
weights $ \leq n$ and $r_1 < r_2 < \ldots < r_t$ and the $\alpha_i$
are integers.

\end{theorem}




\subsection{Factors related to the Schur Multiplicator}
Suppose $F$ is freely generated on a finite set  and $R$ is 
finitely generated as a normal subgroup. 
Then there exists a basis $w_1, w_2, \ldots,
w_t, w_{t+1}, \ldots , $ for $R$ and a basis $X = x_1, x_2, \ldots ,
x_s$  for $F$ such that $ w_i \equiv x^{\al_i} \mod F'$, $\al_i \not = 0$, 
 for $ 1 \leq i \leq
t \leq s$ and $w_j \in F'$ for $j > t$;  see \ref{sec:five}. 
Let $W_1 = w_1, w_2, \ldots, w_t$ and $W_2 = w_{t+1},\ldots $.

Suppose now $r \in R\cap \ga_2(F)$. Then $r = w_1^{\be_1}w_2^{\be_2}
... w_t^{\be_t}w_{t+1}^{\be_{t+1}} .. \mod R'$ (with only a finite number of
non-zero powers). As $r\in \ga_2(F)$ and $w_j \in \ga_2(F)$ for $j \geq t+1$ 
this implies that $\be_i =0$ for $1
\leq i \leq t$. Thus $r$ is generated modulo $R'$ by elements in $W_2$. 

Apply the $U$-construction to $W_1 \cup W_2$ to get a set $U_1 $ which
is part of a free generating set for $R'$. 
\begin{theorem}\label{thm:three}
$R\cap \ga_2(F)$ has free generating set $ W_2 \cup U_1$ and $R'$ has free
generating set $U_1 \cup U$ where $U$ is the set obtained from the $U$
construction on $W_2 \cup U_1$.
\end{theorem}
\begin{proof} We need to show that $R\cap \ga_2(F)$ is generated by
$W_2 \cup U_1$. Consider an element $r\in R\cap \ga_2(F)$. This is a
word $w$  in
$W_1 \cup W_2$. We know that the coefficient sum of any element of
$W_1$ in $w$ is $0$. Collect in $w$ the elements of $W_1$ as described
in the proof of \thmref{one}. Since the coefficient sum of any element
of  $W_1$ in $w$
is $0$ and elements of $U_1$ are formed in the collection process it
is then clear that $w$ is a word in $W_2 \cup U_1$. This set is also
independent.

The $U$ construction on  $ W_1$ and $U_1$ gives the free generators of
$R'$ as required.
\end{proof}   

Suppose now $R$  is generated as a normal subgroup by $S= r_1, r_2,
\ldots, r_n$. Then clearly $S$ generates
$\frac{R}{[R,F]}$.
Then there exist a set $\hat{S} = \hat{r}_1, \hat{r}_2, \ldots,
\hat{r}_n$ equivalent to $S$
such that $ \hat{r}_i \equiv x_i^{\al_i} \mod \ga_2(F)$, $\al_i\not = 0$,
 for $ 1 \leq i \leq s \leq
n$,  and $\hat{r}_i
\in \ga_2(F)$ for $i>s$ where $x_1, x_2, \ldots , x_s$ is part of a free
basis for $F$.  Set $T = \hat{r}_{s+1}, \hat{r}_{s+2}, \ldots \hat{r}_{n}$.

\begin{lemma}\label{thmref:four}
$T$ generates $\frac{R\cap \ga_2(F}{[R,F]}$.
\end{lemma} 
\begin{proof} Consider $r \in R\cap \ga_2(F)$. Then $ r = 
\displaystyle\prod_{i= 1}^{n}
r_i^{\al_i}\mod [R,F]$. Since $r\in R\cap \ga_2(F)$ and $ r_{i}\in
R\cap \ga_2(F)$ for $i >
s$ it follows that $\displaystyle\prod_{i=1}^{s}r_i^{\al_i} \in \ga_2(F)$
from which  it follows that $\al_i = 0$ for $1 \leq i \leq s$. Thus $T$
generates $\frac{R\cap \ga_2(F)}{[R,F]}$.
\end{proof} 

Now from \thmref{three} each $r_i$ for $i > s$ is a product of elements
from $W_2$ modulo $R'$. From this it follows that exists a $T' =
r'_{s+1}, r'_{s+2}, \ldots, r'_{n}$ equivalent to $T$ and a set $\hat{W}_2 =
\hat{w}_{s+1}, \hat{w}_{s+2}, \ldots, $ equivalent to $W_2$
with ${r}'_i \equiv \hat{w}_i^{\be_i} \mod R'$, $\be_i \not = 0$,  for $ s+1 \leq i \leq t
\leq n$ and $r'_{i} \in R'$ for $t+1 \leq i \leq n$.
Set $W= \{ \hat{w}_i^{\be_i}\, | \, s+1 \leq i \leq t\}$. 
Thus:
\begin{theorem} ${W}$ 
generates $\frac{R\cap F'}{[R,F]}$.
 \end{theorem}

\paragraph{\bf References}
\begin{enumerate}
\item P. Hall, {\em The Edmonton notes on Nilpotent groups}, Queen Mary
  College 1970.
\item K. W. Gruenberg, {\em Cohomological Topics in Group Theory}, 
Lecture Notes in Mathematics, vol. 143, Springer-Verlag, Berlin-New York, 1970.
\item Hurley, T.C. \& Ward, M.A. ``Bases for commutator
  subgroups of a free group'' , Proc. RIA, Vol 96A, No. 1, 43-65 (1996). 
\item Magnus,W.,  Karrass, A. , Solitar,D., {\em Combinatorial Group
Theory}, Interscience 1966.
\end{enumerate}

\noindent Ted Hurley \\
Department of Mathematics \\
National University of Ireland, Galway \\
Galway \\
Ireland. \\
email: ted.hurley@nuigalway.ie
\end{document}